\documentclass[11pt]{amsart}
\usepackage{amsfonts,amssymb,amsmath,amsthm}
\usepackage{url}
\usepackage{graphicx} % in the preamble
\usepackage[all]{xy}

\usepackage{enumerate}
\usepackage{shuffle}

\usepackage{amsmath,amsfonts,amsthm,url,color,amssymb}
\usepackage{graphicx}
\usepackage{algpseudocode, algorithm}
\usepackage{hyperref}
\usepackage{todonotes}

\usepackage[norefs,nocites]{refcheck}

\newtheorem{theorem}{Theorem}[section]
\newtheorem{lemma}[theorem]{Lemma}

\newtheorem{proposition}[theorem]{Proposition}
\theoremstyle{definition}
\newtheorem{definition}[theorem]{Definition}

\theoremstyle{remark}
\usepackage{amssymb}
\newtheorem{remark}[theorem]{Remark}
\usepackage{enumerate}

\author{Rafael Bezerra dos Santos}
\address{Departamento de Matem\'{a}tica, Universidade Federal de Minas Gerais, UFMG, Belo Horizonte, MG, 30270-901, Brazil.}
\curraddr{}
\email{rafaelsantos23@mat.ufmg.br}
\thanks{The first author was partially supported by CNPq (Brazil), grant 312058/2025-0, and by FAPEMIG (Brazil), grant RED-00133-21. Corrresponding author.}
\thanks{}

\author{Lucas Reis}
\address{Departamento de Matem\'{a}tica, Universidade Federal de Minas Gerais, UFMG, Belo Horizonte, MG, 30270-901, Brazil.}
\curraddr{}
\email{lucasreismat@mat.ufmg.br}
\thanks{The second author was supported by CNPq (Brazil), grants 310583/2025-0 and 420721/2025-8, and by FAPESP (Brazil), grant 25/21886-4.}

\begin{document}

\title[Identities, central polynomials and cocharacters of $(M_{1, 1}(F), trp)$]{Polynomial identities, central polynomials and cocharacters of $M_2(F)$ with transpose superinvolution} 

\begin{abstract} Let $F$ be a field of characteristic zero and consider $M_2(F),$ the algebra of $2\times 2$ matrices over $F,$ with canonical $\mathbb{Z}_2$-grading and endowed with transpose superinvolution. In this paper, we present the generators of the $T_2^*$-ideal of $*$-identities and of the $T_2^*$-subspace of central $*$-polynomials of $M_2(F).$ As a consequence, we determine the sequences of $*$-codimensions and $\langle n\rangle$-cocharacters of $M_2(F)$. In particular, we prove that the $n$-th $*$-codimension grows like $4^nn^{-1/2}$. 
\end{abstract}

\maketitle

\section{Introduction}

Let $F$ be a field of characteristic zero and let $A$ be an associative algebra over $F$. Let $\{x_1,x_2,\ldots\}$ be a countable set of non commutative variables and let $F\langle X\rangle$ be the free algebra on $X$ over $F$. A polynomial $f(x_1,\ldots, x_n)\in F\langle X\rangle$ is a polynomial identity for the algebra $A$ if $f(a_1,\ldots, a_n)=0$, for all $a_1,\ldots ,a_n\in A$. We say that $A$ is a PI-algebra if $A$ satisfies a non trivial polynomial identity. Describing the set $Id(A),$ the $T$-ideal of $A$, of all polynomials identities of a given algebra $A$ is one of the central and most challenging problem in PI-theory. For example, if $M_n(F)$ is the algebra of $n\times n$ matrices over $F$, we do not know a set of generators for $Id(M_n(F))$ if $n\geq 3.$

In order to ``measure" the size of $Id(A)$, Regev \cite{Regev} introduce the sequence of codimensions of an algebra $A$, given by $c_n(A)=\dim_F \dfrac{P_n}{P_n\cap Id(A)},n\geq 1$, where $P_n$ denotes the set of multilinear polynomials of degree $n$. In particular, he proved that if $A$ is a PI-algebra, then such a sequence is exponentially bounded.

Another invariant that can be attached to an algebra $A$ is the sequence of cocharacters of $A.$ It is well known that the symmetric group $S_n$ acts on $P_n$ on the left and, since $P_n\cap Id(A)$ is invariant under this action, then $\dfrac{P_n}{P_n\cap Id(A)}$ inherits a structure of left $S_n$-module. Thus, the $n$-th cocharacter of $A$, denoted by  $\chi_n(A)$, is the $S_n$-character of $\dfrac{P_n}{P_n\cap Id(A)}$.

Let $f\in F\langle X\rangle$. We say that $f$ is a central polynomial for an algebra $A$ if it has no constant term and every evaluation of $f$ at elements from $A$ lies in the center of $A$. The set of all central polynomials of an algebra $A$ is denoted by $Id^z(A)$. It is clear that every polynomial identity is also a central polynomial and, if $f$ is a central polynomial, then the commutator $[f,x]$ belongs to $Id(A)$.

In this paper, we are interested in the study of the algebra $M_2(F)$. In 1981, Drensky \cite{Drensky1}, described the generators of  $Id(M_2(F))$. In the same year, Formanek et.al \cite{Formanek} provided the decomposition of the $n$-th cocharacter of $M_2(F)$. In the context of central polynomials, Okhitin \cite{Okt} determined a generating set for $Id^z(M_2(F))$ in 1988.

The same can be made in the context of algebras with additional structures. An involution $*$ on an algebra $A$ is an antiautomorphism of order at most 2. The concepts of identities, codimensions, cocharacters and central polynomials can be defined analogously to the ordinary case. In this context, Levchenko \cite{Lev} determined the generators of the $T^*$-ideal of $*$-identities of $M_2(F)$ when is endowed with transpose or symplectic involution. The sequence of $*$-cocharacters of $M_2(F)$ was provided by Drensky and Giambruno in \cite{Giam1} and the generators of the space of central 
$*$-polinomials were presented by Brandão and Koshlukov in \cite{Plamen2}. Analogous results were established in the $\mathbb{Z}_2$-graded case \cite{Plamen3} and, more recently, in the more general context of $*$-superalgebras \cite{Ana}.

In this paper we consider the algebra $M_2(F)$ with the canonical $\mathbb Z_2$-grading, endowed with transpose superinvolution: this is the unique nontrivial superinvolution on $M_2(F)$. We show that the $T_2^*$-ideal of $*$-identities of this algebra is essentially generated by $6$ identities: see Theorem~\ref{thm:main}. As a consequence of the latter, we derive explicit formulas for its sequences of $*$-codimensions and $\langle n\rangle$-cocharacters: see Proposition~\ref{prop:cod} and Theorem~\ref{thm:coc}. In particular, Proposition~\ref{prop:cod} reveals that the $n$-th $*$-codimension grows like $4^nn^{-1/2}$. Finally, in Theorem~\ref{thm:main2} we provide a set of generators for the $T_2^*$-subspace of central $*$-polynomials for the algebra: this set of generators contains only one polynomial that is not a $*$-identity.

We end this section with the structure of the paper. In Section 2 we briefly recall definitions and results on algebras with additional structures. In Section 3 we provide basic machinery on the PI-Theory of algebras with superinvolution, and also present two auxiliary lemmas. Finally, in Sections 4 and 5 we state and prove our results.

\section{Algebras with superinvolutions}
	
	Let $F$ be a field of characteristic zero and let $A$ be an associative algebra over $F$. We say that $A$  is a {\em superalgebra} (or a $\mathbb{Z}_2$-graded algebra) if $A$ can be written as a sum of 2 subspaces $A=A^{(0)}+A^{(1)}$ such that $A^{(0)}A^{(0)}+A^{(1)}A^{(1)} \subseteq A^{(0)}$ and $A^{(0)}A^{(1)}+A^{(1)}A^{(0)}\subseteq A^{(1)}$. The elements from $A^{(i)}$ with $i=0,1$ are called homogeneous of degree $i$. We denote by $|a|$ the homogeneous degree of a homogeneous element $a\in A^{(0)}\cup A^{(1)}.$ 
	
	A linear map $*:A\to A$ such that $(A^{(0)})^*=A^{(0)}$ and $(A^{(1)})^*=A^{(1)}$ is called a superinvolution on $A$ if $(a^*)^*=a$ and $(ab)^*=(-1)^{|a||b|}b^*a^*$. A superalgebra endowed with a superinvolution is called a $*$-algebra. Notice that if $(A^{(1)})^2=\{0\}$, then a superinvolution is just a graded involution and, in this case, $A$ is a $*$-superalgebra.    
	
	Since char$(F)=0$ and $(A^{(i)})^*=A^{(i)}$ for
    $i=0,1$, each homogeneous component can be written as $A^{(i)}=(A^{(i)})^+ +(A^{(i)})^-$, where $(A^{(i)})^+=\{a\in A^{(i)}:a^*=a\}$ and $(A^{(i)})^-=\{a\in A^{(i)}:a^*=-a\}$. Therefore, if $A$ is endowed with a superinvolution, then $A$ can be written as $$A=(A^{(0)})^+ +(A^{(0)})^- + (A^{(1)})^+ + (A^{(1)})^-.$$
	
	Let $(A,\psi_1)$ and $(B,\psi_2)$ algebras with superinvolutions $\psi_1$ and $\psi_2$, respectively. We say that $A$ and $B$ are isomorphic, as a $*$-algebras, if there exists an isomorphism of algebras $\phi:A\to B$ such that $\phi(A^{(i)})=B^{(i)},i=0,1,$ and $\phi(a^{\psi_1})=\phi(a)^{\psi_2}$, for all $a\in A$.
	
	An ideal $I$ of an algebra $A$ with superinvolution $*$ is a $*$-ideal if $I=I^{(0)}+I^{(1)}$, where $I^{(i)}=I\cap A^{(i)}, i=0,1,$ and $I^*=I$. We say that an algebra $A$ with superinvolution $*$ is $*$-simple, if $A^{2}\neq \{0\}$ and $A$ has no nontrivial $*$-ideals.
	
	Next, we shall present the classification of simple $*$-algebras over an algebraically closed field of characteristic different from 2. Let $M_n(F)$ be the algebra of $n\times n$ matrices over $F$. If $n=k+l, k\geq l\geq 0,$ then $M_n(F)$ becomes a superalgebra, denoted by $M_{k,l}(F)$, with grading $$(M_{k,l}(F))^{(0)}=\left\{\begin{pmatrix}
		A&0\\0&B
	\end{pmatrix}: A\in M_k(F), B\in M_l(F) \right\}$$ and $$(M_{k,l}(F))^{(1)}=\left\{\begin{pmatrix}
	0&C\\D&0
	\end{pmatrix}: C\in M_{k\times l}(F), D\in M_{l\times k}(F) \right\}.$$
	
	Up to $*$-isomorphism, Racine \cite{Racine} proved that $M_{k,l}(F)$ can be endowed with only 2 superinvolutions:
	
	\begin{itemize}
		\item[1.] the {\em transpose superinvolution}, denoted by $trp$, defined  if $k=l$, and given by $$\begin{pmatrix}
			A&C\\D&B
		\end{pmatrix}^{trp}=\begin{pmatrix}
		B^t&-C^t\\ D^t&A^t
		\end{pmatrix},$$ where $t$ denotes the usual transpose;
		\item[2.] the {\em ortosymplectic superinvolution}, denoted by $osp$, defined if $l=2m$ is even, and given by
		$$\begin{pmatrix}
			A&C\\D&B
		\end{pmatrix}^{osp}=\begin{pmatrix}
		I_k&0\\0&Q
		\end{pmatrix}^{-1}\begin{pmatrix}
		A&-C\\D&B 
		\end{pmatrix}^{t}\begin{pmatrix}
		I_k&0\\0&Q
		\end{pmatrix}=\begin{pmatrix}
		A^t&D^tQ\\ QC^t& -QB^tQ
		\end{pmatrix},$$ where $Q=\begin{pmatrix}
		0&I_m\\-I_m&0
		\end{pmatrix}$ and for each $j\ge 1$, $I_j$ denotes the $j\times j$ identity matrix.
		
	\end{itemize}
	
	Now, if $A$ is a superalgebra, we denote by $A^{sop}$ the algebra with the same structure of superalgebra of $A$ and product given on homogeneous elements $a,b\in A^{sop}$ by $a\circ b=(-1)^{|a||b|}ba$. The algebra $R=A\oplus A^{sop}$ is a superalgebra with induced grading and can be endowed with the exchange superinvolution $(a,b)^*=(b,a)$. For instance, if we consider the superalgebra $Q(n)=M_n(F+cF),c^2=1$, with grading $Q_n^{(0)}=M_n(F)$ and $Q_n^{(1)}=cM_n(F)$, then $Q(n)\oplus Q(n)^{sop}$ is a $*$-algebra with exchange superinvolution. We have the following result.
	
	\begin{theorem}[\cite{Bahturin, Shest, Racine}]
		Let $A$ be a finite dimensional simple $*$-algebra over an algebraically closed field of characteristic different from 2. Then $A$ is isomorphic, as a $\ast$-algebra, to the following:
		\begin{itemize}
			\item[1.] $M_{k,l}(F)$ with the transpose or ortosymplectic superinvolution;
			\item[2.] $M_{k,l}(F)\oplus M_{k,l}(F)^{sop}$, with exchange superinvolution;
			\item[3.] $Q(n)\oplus Q(n)^{sop}$, with exchange superinvolution.
		\end{itemize}
	\end{theorem}
	
	We notice that the only nontrivial superinvolution on $M_2(F)$ is the transpose superinvolution, with grading $M_{1,1}(F)$. In this case, if $A=M_{1,1}(F)$, then $(A^{(0)})^+=span_F\{e_{11}+e_{22}\}$, $(A^{(0)})^-=span_F\{e_{11}-e_{22}\}$, $(A^{(1)})^+=span_F\{e_{21}\}$ and $(A^{(1)})^-=span_F\{e_{12}\}$, where $e_{ij}$ denotes the usual matrix unity.
	
	%\textcolor{red}{remover isso, pois estará na introduction} In this paper, our main goal is to describe the $*$-identities, the central $*$-polynomials and the $\langle n\rangle$-cocharacter of $M_{1,1}(F)$ with transpose superinvolution, as we will see below.
	
	\section{The free algebra with superinvolution}
	
	Let $X = \{x_{1}, x_{2}, \ldots \}$ be a countable set of non-commuting variables. We write $X$ as a disjoint union of four subsets $\displaystyle X = Y_{0} \cup Y_{1} \cup Z_{0} \cup Z_{1}$ and we denote by $\mathcal{F} = F \langle X|* \rangle$ the free associative $*$-algebra on $X$ over $F$, where $Y_{0} = \{y_{1,0}, y_{2,0}, \ldots \}$ is the set of symmetric variables of degree $0$, $Y_{1} = \{y_{1,1}, y_{2,1}, \ldots \}$ is the set of symmetric variables of degree $1$, $Z_{0} = \{z_{1,0}, z_{2,0}, \ldots \}$ is the set of skew symmetric variables of degree  $0$ and $Z_{1} = \{z_{1,1}, z_{2,1}, \ldots \}$ is the set of skew symmetric variables of degree  $1$. The elements of $\mathcal{F}$ are called $*$-polynomials. Since $\mathcal{F}$ is a $\ast$-algebra, it can be written as $\mathcal{F}=(\mathcal{F}^{(0)})^+ +(\mathcal{F}^{(0)})^-+(\mathcal{F}^{(1)})^+ +(\mathcal{F}^{(1)})^-$.
	
	We will denote by $[x,y]=xy-yx$ the commutator of the variables $x,y\in X$ and by $x\circ y=xy+yx$ the Jordan product of the variables $x,y\in X$. The following remark is verified by direct computations.

\
    
    \

	\begin{remark}\label{consequences}
		\
		
		\begin{itemize}
			\item[1.] $[y_{1,0},z_{2,0}], [y_{1,1},y_{2,1}],[z_{1,1},z_{2,1}],y_{1,0}\circ y_{2,0},y_{1,1}\circ z_{2,1}\in (\mathcal{F}^{(0)})^+$;
			\item[2.] $[y_{1,0},y_{2,0}], [y_{1,1},z_{2,1}], y_{1,0}\circ z_{2,0}, y_{1,1}\circ y_{2,1},z_{1,1}\circ z_{2,1}\in (\mathcal{F}^{(0)})^-$;
			\item[3.] $[y_{1,0},z_{2,1}], [z_{1,0},y_{2,1}],y_{1,0}\circ y_{2,1}, z_{1,0}\circ z_{2,1}\in (\mathcal{F}^{(1)})^+$;
			\item[4.] $[y_{1,0},y_{2,1}], [z_{1,0},z_{2,1}],y_{1,0}\circ z_{2,1}, z_{1,0}\circ y_{2,1}\in (\mathcal{F}^{(1)})^-$.
		\end{itemize}
	\end{remark}
	
	We say that $$f = f(y_{1,0},  \ldots , y_{m,0}, y_{1,1},  \ldots, y_{n,1}, z_{1,0}, \ldots, z_{p,0}, z_{1,1},  \ldots, z_{q,1}) \in \mathcal{F}$$
	is a $*$-identity of a $*$-algebra $A$ if $$f(a_{1,0}, \ldots, a_{m,0},a_{1,1}, \ldots , a_{n,1}, b_{1,0}, \ldots , b_{p,0}, b_{1,1}, \ldots , b_{q,1}) = 0,$$ 
	for all $a_{1,0}, \ldots, a_{m,0} \in (A^{(0)})^{+}$, $a_{1,1}, \ldots, a_{n,1} \in (A^{(1)})^{+}$, $b_{1,0}, \ldots, b_{p,0} \in (A^{(0)})^{-}$ and $b_{1,1}, \ldots , b_{q,1} \in (A^{(1)})^{-}$. In this case, we write $f\equiv 0$ on $A$. The set of all polynomial $*$-identities satisfied by the algebra $A$ 
	$$
	Id^*(A) = \{f \in \mathcal{F}: f \equiv 0 \mbox{ on } A \}
	$$ 
	is a $T^{*}_{2}$-ideal of $\mathcal{F}$, i.e., an ideal of $\mathcal{F}$ invariant under all graded endomorphisms of $\mathcal{F}$ commuting with the superinvolution $*$.
	
	Given a set of $*$-polynomials $S \subseteq \mathcal{F}$, we denote by $\langle S \rangle_{T_2^*}$ the $T_2^*$-ideal of $\mathcal{F}$ generated by the set $S$. Moreover, we say that a set of $*$-polynomials $S'$ is a consequence of $S$ if $S' \subseteq \langle S \rangle_{T_2^*}$. Since char$(F) =0$, each polynomial of $Id^*(A)$ is
	equivalent to a system of multilinear $*$-identities. We denote by
	$$
	P_{n}^{*} = \textrm{span}\{w_{\sigma(1)} \cdots w_{\sigma(n)}: w_{i} \in \{y_{i,0},y_{i,1},z_{i,0},z_{i,1}\}, \sigma \in S_{n}\},
	$$ 
	the space of multilinear $*$-polynomials of degree $n$ in the variables $y_{1,0},  \ldots , y_{n,0}, \\ y_{1,1}, \ldots, y_{n,1}, z_{1,0}, \ldots, z_{n,0}, z_{1,1},  \ldots, z_{n,1}$. Notice that $\dim_FP_n^*=4^nn!$. The dimension of the quotient space
	$$P_n^{*}(A)=\displaystyle \frac{P_n^{*}}{P_n^{*}\cap Id^*(A)}$$ is called the $n$-th $*$-codimension of $A$ and it is denoted by  $c_n^*(A)$.
	
	\begin{remark}
		If $c_n(A)$ denotes the ordinary codimension of a $*$-algebra $A$, then $0\leq c_n(A)\leq c_n^*(A) \leq 4^nc_n(A)$. As a consequence, by \cite{Regev}, if $A$ satisfies a nontrivial ordinary polynomial identity, then $c_n^*(A)$ is exponentially bounded.
	\end{remark}
	
	In \cite{Ninni}, Ioppolo proved the following relations about the sequence of $*$-codimensions of a $*$-algebra. 
	
	\begin{theorem}
		Let $F$ be a field of characteristic zero and $A$ be a $*$-algebra over $F$. Then, there exist constants $C_1>0, C_2, t_1, t_2$ and an integer $d\geq 0$ such that $$C_1n^{t_1}d^n\leq c_n^*(A)\leq C_2n^{t_2}d^n.$$
	\end{theorem}
	
	The integer $d\geq 0$ in the previous theorem is called $*$-exponent of $A$ and it is denoted by exp$^*(A)$.
	
	Now, let $n\geq 1$ and consider $n=n_1+n_2+n_3+n_4$, where each $n_i$ is a nonnegative integer. This is a composition of $n$ and we write $\langle n\rangle=(n_1,n_2,n_3,n_4)$. Given a composition $\langle n\rangle=(n_1,n_2,n_3,n_4)$ of $n$, let $P_{\langle n\rangle}\subseteq P_n^*$ be the vector space of multilinear $*$-po\-ly\-no\-mi\-als in which $n_1$ variables are symmetric of homogeneous degree 0, $n_2$ variables are skew of homogeneous degree 0, $n_3$ variables are symmetric of homogeneous degree 1 and $n_4$ variables are skew of homogeneous degree 1. It is clear that there exist $\binom{n}{{\langle n\rangle}}=\frac{n!}{n_1!n_2!n_3!n_4!}$ subspaces of $P_n^*$ that are isomorphic to $P_{\langle n\rangle}$ and, as a consequence, $$ P_n^*\cong \bigoplus_{{\langle n\rangle}}\binom{n}{{\langle n\rangle}}P_{{\langle n\rangle}},$$  where the sum runs over all compositions $\langle n\rangle$ of $n$. 
	
	Therefore, if $P_{\langle n\rangle}(A)=\dfrac{P_{\langle n\rangle}}{P_{\langle n\rangle}\cap Id^*(A)}$ and $c_{\langle n\rangle}(A)=\dim_F(P_{\langle n\rangle}(A))$, then \begin{equation}\label{codim}
		c_n^*(A)=\displaystyle \sum_{\langle n\rangle}\binom{n}{{\langle n\rangle}}c_{\langle n\rangle}(A).
	\end{equation}
	
	Given a composition $\langle n\rangle=(n_1,n_2,n_3,n_4)$ of $n$, we denote by $S_{\langle n \rangle}=S_{n_1}\times \cdots \times S_{n_4}$ the direct product of four symmetric groups $S_{n_i},1\leq i\leq 4$. Given $\langle n \rangle$, the group $S_{\langle n \rangle}$ acts on the left on $P_{\langle n \rangle}$ as follows: $S_{n_1}$ permutes the variables $y_{1,0}, \ldots, y_{n_1,0}$, $S_{n_2}$ permutes the variables $z_{1,0},\ldots, z_{n_2,0}$, and so on. In this way, $P_{\langle n \rangle}$ becomes a left $S_{\langle n \rangle}$-module. Notice that $P_{\langle n\rangle}\cap Id^*(A)$ is invariant under this action and so $P_{\langle n\rangle}(A)$ is a left $S_{\langle n\rangle}$-module. Its $S_{\langle n\rangle}$-character is denoted by $\chi_{\langle n\rangle}(A)$ and is called the $\langle n\rangle$-cocharacter of $A$.

	Given a composition $\langle n\rangle$ of $n$, a multipartition of $\langle n\rangle$, denoted by $\langle \lambda \rangle\vdash \langle n\rangle$, is a sequence of partitions $\langle \lambda\rangle=(\lambda(1),\ldots, \lambda(4))$ such that $\lambda(i)\vdash n_i,i=1,\ldots, 4$. It is well know that there exists a one-to-one correspondence between partitions of $n$ and irreducible $S_n$-characters. Thus, if $\lambda\vdash n$ , we denote by $\chi_{\lambda}$ the corresponding irreducible $S_n$-character.
	Since $F$ is a field of characteristic zero, by complete reducibility, one can write \begin{equation}\label{ncocharacter}
		\chi_{\langle n\rangle}(A)=\displaystyle \sum_{\langle \lambda \rangle\vdash \langle n\rangle} m_{\langle \lambda \rangle}\chi_{\langle \lambda \rangle},
	\end{equation} where $m_{\langle \lambda \rangle}$ is the corresponding multiplicity of the irreducible $S_{\langle n \rangle}$-character $\chi_{\langle \lambda \rangle}=\chi_{\lambda(1)}\otimes \cdots \otimes \chi_{\lambda(4)}$. In order to compute the multiplicities $m_{\langle \lambda \rangle}$ of the
	irreducible characters appearing in (\ref{ncocharacter}) we use the
	re\-presentation theory of the general linear group $GL_n$.

	Let $F_{n,*}=F\langle y_{1,0},\ldots y_{n,0},z_{1,0},\ldots z_{n,0}, y_{1,1},\ldots y_{n,1}, z_{1,0},\ldots z_{n,1}\rangle$ be the free associative $*$-algebra of rank $n$ and consider $U_1=span_F\{{y_{1,0},\ldots y_{n,0}}\}$, $U_2=span_F\{{z_{1,0},\ldots z_{n,0}}\}$, $U_3=span_F\{{y_{1,1},\ldots y_{n,1}}\}$ and $U_4=span_F\{{z_{1,1},\ldots z_{n,1}}\}.$ There is a natural left action of the group $GL(U_1)\times \cdots\times GL(U_4)\cong GL_n^4$ on the space $U_1\oplus \cdots\oplus U_4$ and we can extend this action diagonally to get an action on $F_{n,*}$. Notice that, if $A$ is a $*$-algebra, then $F_{n,*}\cap Id^*(A)$ is invariant under this action.
	
	Considering $F^m_{n,*}$ the space of all multihomogenous $*$-polynomials of degree $m$ in $F_{n,*}$, the quotient space $F^m_{n,*}(A)=\dfrac{F^m_{n,*}}{F^m_{n,*}\cap Id^*(A)}$ is a $GL_n^4$-module and its $GL_n^4$-character is denoted by $\psi_n^{*}(A)$.

	In what follows, the reader may consult \cite{livro} for formal definitions and proofs of the results. There exists a one-to-one correspondence between $GL_n^4$-characters and multipartitions $\langle \lambda\rangle\vdash \langle n\rangle$, where each $\lambda(i)$ has at most $m$ parts. Therefore,\begin{equation}\label{cocharacter}
	\psi_n^*(A)=	\displaystyle \displaystyle{\sum_{ \tiny{\begin{array}{ccc}
						\langle \lambda \rangle \vdash \langle n\rangle\\
						h(\lambda(i))\leq m
		\end{array}}}} \bar{m}_{\langle \lambda \rangle}\psi_{\langle \lambda \rangle},
	\end{equation} where $\bar{m}_{\langle \lambda \rangle}$ is the corresponding multiplicity of the irreducible $GL_n^4$-character $\psi_{\langle \lambda \rangle}$ associated to the multipartition $\langle \lambda \rangle$ and $h(\lambda(i))$ denotes the height of the Young diagram corresponding to $\lambda(i)$ for each $1\le i\le 4$.
	
	\ A Young multitableau $T_{\langle \lambda \rangle}$ of shape $\langle \lambda \rangle$ is a 4-tuple $(T_{\lambda(1)},\ldots, T_{\lambda(4)})$ of Young tableaux. Each multitableau $T_{\langle \lambda \rangle}$ can be associated to a polynomial $f_{
	T_{\langle \lambda \rangle}},$ called {\em highest weight vector}. Moreover, an irreducible submodule of $F^m_{n,*}(A)$ corresponding to the multipartition $\langle \lambda \rangle$ is generated by a non-zero polynomial $f_{\langle \lambda \rangle}$ which can be written as a linear combination of highest weight vectors $f_{
	T_{\langle \lambda \rangle}}.$
	
	\begin{remark}\label{multiplicities}
		
		\
		
		\begin{itemize}
			\item[1.] We have that  $m_{\langle \lambda \rangle}=\bar{m}_{\langle \lambda \rangle}$, for all multipartitions $\langle \lambda \rangle$ such that $h(\lambda(i))\leq m, i=1,\ldots,4$;
			
			\item[2.] For the decomposition in Eq.~\eqref{cocharacter}, we have $\bar{m}_{\langle \lambda \rangle}\neq 0$ if and only if there exists a multitableau $T_{\langle \lambda \rangle}$ such that the corresponding highest weight vector satisfies $f_{T_{\langle \lambda \rangle}}\notin Id^{*}(A)$. Moreover, $\bar{m}_{\langle \lambda \rangle}$ is the maximal
			number of linearly independent highest weight vectors $f_{T_{\langle \lambda \rangle}}$ in $F^m_{n,*}(A)$;
			\item[3.] If $A$ is a $*$-algebra, $A=(A^{(0)})^+ +(A^{(0)})^- + (A^{(1)})^+ + (A^{(1)})^-,$ and $\dim_F((A^{(i)})^{\epsilon})=d_i^{\epsilon}, i\in \{0,1\},\epsilon\in\{+,-\}$, then $\bar{m}_{\langle \lambda \rangle}= 0$ if $h(\lambda(1))>d_0^+, h(\lambda(2))>d_0^-,h(\lambda(3))>d_1^+$ and $h(\lambda(4))>d_1^-.$
		\end{itemize}
		
	\end{remark}
	
	Next, we will deal with central polynomials. Given an algebra with superinvolution $A$, a $*$-polynomial $f\in \mathcal{F}$ is a central $*$-polynomial for $A$ if it has no constant term and any admissible evaluation on $f$ at elements of $A$ belongs to the center of $A$. It is clear that a $*$-polynomial identity of $A$ is a central $\ast$-polynomial of $A$. The set $$Id^{*,z}(A)=\{f\in \mathcal{F}:\mbox{ $f$ is a central $*$-polynomial for $A$}\}$$ is not, in general, a $T_2^*$-ideal of $\mathcal{F}$. However, it is a $T_2^*$-subspace of $\mathcal{F}$, i.e. a subspace invariant under all graded endomorphisms of $\mathcal{F}$ commuting with the superinvolution $*$. If $S\subset \mathcal{F}$, we denote by $\langle S\rangle^{T_2^*}$ the $T_2^*$-subspace generated by $S$. Notice that $\langle S\rangle^{T_2^*}$ is generated, as a subspace, by $f(g_1,\ldots,g_n)$, where $f\in S$ and each $g_i$ has the same homogeneous degree and symmetry
as the variable it replaces.

     \begin{remark}\label{rem:center}
         For $n\geq 1$, the center of $M_n(F)$ is isomorphic to $F$. More precisely, $Z(M_n(F))=\{\alpha I_n:\alpha \in F\},$ where $I_n$ denotes the $n\times n$ identity matrix.
     \end{remark}

	 In the next section, our main goal is to determine the sets $Id^*(A)$ and $Id^{*,z}(A)$, the explicit decomposition of $\chi_{\langle n\rangle}(A)$ and to provide asymptotics for $c_n^{*}(A)$, when $A$ is the algebra $M_{1,1}(F)$ endowed with transpose superinvolution.
	 
	We end this section with some relevant auxiliary lemmas. In the next results, $F$ is a field of arbitrary characteristic.
	
		\begin{lemma}\label{lem:vanish}
		Let $F$ be an infinite field and let $f$ be a polynomial over $F$ in $n$ commuting variables. If $f$ vanishes entirely on $F^n$, then $f$ equals the zero polynomial.
	\end{lemma}
	
	Given $\mathcal{F}=F \langle X|* \rangle$, we denote by $\mathcal{F}^C=F[X|*]$ the free associative and commutative $*$-algebra on $X$ over $F$.
	
	\begin{definition}
		
		Let $f\in \mathcal{F}$. We denote by $f^C\in \mathcal{F}^C$ the polynomial induced by $f$ through the relations $x_ix_j=x_jx_i$ for each $x_k\in\{y_{k,0},z_{k,0},y_{k,1},z_{k,1}\}$ with  $k=i,j$.
	\end{definition}
	For instance, if $f(y_{1,0}, z_{1,0})=y_{1,0}z_{1,0}+z_{1,0}y_{1,0}\in \mathcal{F}$, then $f^C(y_{1,0}, z_{1,0})=2y_{1,0}z_{1,0}\in \mathcal{F}^C$. The polynomial $f^C$ can be zero even if $f\in \mathcal F$ is nonzero: $f(y_{1, 0}, z_{1, 0}, z_{1, 1})=y_{1,0}z_{1,0}z_{1, 1}-z_{1,0}z_{1, 1}y_{1,0}\in \mathcal F$ satisfies $f(y_{1, 0}, z_{1, 0}, z_{1, 1})^C=0\in \mathcal F^C$. The following lemma, crucial in the proofs of our main results, provides a sufficient condition under which the vanishing of $f^C$ implies the vanishing of $f$.
	
\begin{lemma}\label{lem:com}
		Let $F$ be a field and let $f\in \mathcal{F}$ be a polynomial not admitting two monomials $M_1, M_2$ with $(M_1)^C=(M_2)^C$. If $f^C\in \mathcal{F}^C$ is the zero polynomial, the same holds for 
		$f$. In particular, if $F$ is infinite, $f$ contains $n$ distinct variables and $f^C$ vanishes entirely on $F^n$, then $f$ is the zero polynomial.
	\end{lemma}
	\begin{proof}
		The first statement follows directly by induction on the number of monomials in the expansion of $f$. The second statement follows directly by the first one and Lemma~\ref{lem:vanish}.
	\end{proof}

	\section{$*$-polynomial identities, $*$-codimension and the $\langle n \rangle$-cocharacter}
	For the rest of the paper, $\mathbb M$ will denote the $*$-algebra $M_{1,1}(F)$ endowed with transpose superinvolution. In this section, we will present the generators of the $T_2^*$-ideal $Id^*(\mathbb M)$, compute the $*$-codimension $c_n^*(\mathbb M)$ and exhibit the decomposition of the $\langle n\rangle$-cocharacter of $\mathbb M$. In the proof of our main results we will work with special monomials in $\mathcal F$. To keep the notation concise and avoid long expressions, we introduce the following notation.
    
	\begin{definition}\label{def:basic}

    \
    
		\begin{enumerate}
		    \item For subsets $S, T\subseteq \{1, \ldots, n\}$, we set $$(YZ)_{S, T}=y_{s_1, 0}\cdots y_{s_j, 0}z_{t_1, 0}\cdots z_{t_k, 0},$$ 
		where $(s_i)_{1\le i\le j}$ and $(t_i)_{1\le i\le k}$ denote the increasing sequence of the elements in $S$ and $T$, respectively. 

        \item For $U, V\subseteq \{1, \ldots, n\}$ with $|U|=|V|=j>0$, set 
		$$(Y|Z)_{U, V}=y_{u_1, 1}z_{v_1, 1}\cdots y_{u_j, 1}z_{v_j, 1},$$
		where $(u_i)_{1\le i\le j}$ and $(v_i)_{1\le i\le j}$ denote the increasing sequence of the elements in $U$ and $V$, respectively. We define $(Z|Y)_{U, V}$ in a similar way.

        \item For $U, V\subseteq \{1, \ldots, n\}$ with $|U|=j+1$ and $|V|=j\ge 0$, set 
		$$\{Y|Z\}_{U, V}=y_{u_1, 1}z_{v_1, 1}\cdots y_{u_j, 1}z_{v_j, 1}y_{u_{j+1}},$$
		where $(u_i)_{1\le i\le j+1}$ and $(v_i)_{1\le i\le j}$ denote the increasing sequence of the elements in $U$ and $V$, respectively. We define $\{Z|Y\}_{U, V}$ in a similar way when $|V|=j+1$ and $|U|=j\ge 0$.
		\end{enumerate}
	\end{definition}
	
	We recall that $(\mathbb M^{(0)})^+=span_F\{e_{11}+e_{22}\}$, $(\mathbb M^{(0)})^-=span_F\{e_{11}-e_{22}\}$, $(\mathbb M^{(1)})^+=span_F\{e_{21}\}$ and $(\mathbb M^{(1)})^-=span_F\{e_{12}\}$. In what follows, we provide a set of generators for $Id^*(\mathbb M)$.

	\begin{theorem}\label{thm:main}
		Let $F$ be a field of characteristic zero. Then the $T_2^*$-ideal $I=Id^*(\mathbb M)$ is generated by the following polynomials:
		\begin{enumerate}
			\item $[y_{1, 0}, x]$ with $x\in \{y_{2,0},
			y_{1, 1}, z_{1, 0}, z_{1, 1}\}$;
			\item $y_{1, 1}y_{2, 1}$;
			\item $z_{1, 1}z_{2, 1}$;
			\item $[z_{1, 0}, z_{2, 0}]$;
			\item $z_{1, 0}\circ z_{1, 1}$;
			\item $z_{1, 0}\circ y_{1, 1}$.
			% \item $y_{1, 1}z_{1, 1}y_{2, 1}=y_{2, 1}z_{1, 1}y_{1, 1}$.
		\end{enumerate}
	\end{theorem}
	
	\begin{proof}
		Let $J$ be the $T_2^*$-ideal generated by the above polynomials. It is direct to verify that $J\subseteq I$. Now let $f\in I$. Since $F$ has characteristic $0$, we can assume that $f$ is multilinear of degree $n\ge 1$. From the identities in (1), (4), (5) and (6), we see that $f\equiv P\pmod J$, where $P$ is a multilinear polynomial of degree $n$ with the following property: 
		each monomial of $P$ is of the form $\alpha X_0X_1$, where
		$X_1$ is a monomial of degree $k\ge 0$ in the variables of homogeneous degree one and
		$X_0=(YZ)_{S, T}$ for some $S, T\subseteq \{1, \ldots, n\}$ with $|S|+|T|=n-k$. Moreover, from the identities (2) and (3), we can also assume that the monomial $X_1$ alternates between the two types of variables of homogeneous degree one. Since 
		$$y_{1, 1}\circ [z_{1, 1}, y_{2, 1}]=y_{1, 1}z_{1, 1}y_{2, 1}-y_{1, 1}y_{2, 1}z_{1, 1}+z_{1, 1}y_{2, 1}y_{1, 1}-y_{2, 1}z_{1, 1}y_{1, 1},$$
		identities (2) and (3) imply 
		$$y_{1, 1}\circ [z_{1, 1}, y_{2, 1}]\equiv y_{1, 1}z_{1, 1}y_{2, 1}-y_{2, 1}z_{1, 1}y_{1, 1}\pmod J.$$
		Now, by Remark \ref{consequences}, we have that $ [z_{1, 1}, y_{2, 1}]\in (\mathcal{F}^{(0)})^-$. Hence identity (6) entails that $y_{1, 1}\circ [z_{1, 1}, y_{2, 1}]\in J$. In conclusion,  
		$$y_{1, 1}z_{1, 1}y_{2, 1}\equiv y_{2, 1}z_{1, 1}y_{1, 1}\pmod J.$$
		In a similar way we obtain $z_{1, 1}y_{1, 1}z_{2, 1}\equiv z_{2, 1}y_{1, 1}z_{1, 1}\pmod J.$
		From these two consequences we can also assume that, among the variables of the same type in $X_1$, their indices are ordered in increasing order. 
		
		With Definition~\ref{def:basic} in mind, these observations lead to the following: the polynomial $P$ is of the form $P_0+P_1+P_2+P_3+P_4$, where $P_0$ only contains monomials of the form $\alpha_0\cdot (YZ)_{S, T}$, $P_1$ only contains monomials of the form $\alpha_1\cdot (YZ)_{S, T}\cdot (Y|Z)_{U, V}$, $P_2$ only contains monomials of the form $\alpha_2\cdot (YZ)_{S, T}\cdot (Z|Y)_{U, V}$, $P_3$ only contains monomials of the form $\alpha_3\cdot (YZ)_{S, T}\cdot \{Y|Z\}_{U, V}$ and $P_4$ only contains monomials of the form $\alpha_4\cdot (YZ)_{S, T}\cdot  \{Z|Y\}_{V, U},\alpha_i\in F$. We further write $P_i=P_{i, +}+P_{i, -}$, where $P_{i, +}$ and $P_{i, -}$ are the sums of the monomials of $P_i$ containing an even and odd number of variables $z_{i, 0}$, respectively. We want to prove that each $P_{i, \pm}$ vanishes. Our idea is to take suitable evaluations at the variables and employ Lemma~\ref{lem:com}. From construction, it is clear that the polynomials $P_{i, \pm}$ and $P_{i, +}\pm P_{i, -}$ do not contain two monomials $M_1, M_2$ with $M_1^C=M_2^C$. 
		
		From hypothesis, $P$ is an identity. Take $y_{i, 1}= z_{i, 1}=0$, $y_{i, 0}=a_i(e_{11}+e_{22})$ and $z_{i, 0}=b_i (e_{11}-e_{22})$ with $a_i, b_i\in F$. We obtain
		$$(Q_{0, +}+Q_{0, -})e_{1 1}+(Q_{0, +}-Q_{0, -})e_{2 2}=0,$$
		where $Q_{0, \pm}\in F$ is the image of $P_{0, \pm}^C$ at the elements $a_i, b_i$. Since $e_{11}$ and $e_{22}$ are linearly independent, we obtain $Q_{0, +}+Q_{0, -}=Q_{0, +}-Q_{0, -}=0$ and so $Q_{0, +}=Q_{0, -}=0$. From Lemma~\ref{lem:com}, both $P_{0, +}$ and $P_{0, -}$ are zero. In conclusion, $P_0$ equals zero.
		
		Now take $y_{i, 0}=a_i(e_{11}+e_{22}), z_{i, 0}=b_i(e_{11}-e_{22}), y_{i, 1}=c_i e_{21}$ and $z_{i, 1}=d_i e_{12}$ for each $1\le i\le n$.
		Similarly to the previous computation, we obtain 
		$$(Q_{1, +}-Q_{1, -})e_{22}+(Q_{2, +}+Q_{2, -})e_{11}+(Q_{3, +}-Q_{3, -})e_{21}+(Q_{4, +}+Q_{4, -}) e_{12}=0,$$
		where $Q_{k, \pm}\in F$ is the image of $P_{k, \pm}^C$ at the elements $a_i, b_i, c_i, d_i$. 
		Again, as the elements $e_{ij}$ are linearly independent, we obtain $Q_{k, +}+(-1)^kQ_{k, -}=0$ for each $1\le k\le 4$. From Lemma~\ref{lem:com}, we obtain $P_{k, +}+(-1)^kP_{k, -}=0$. From construction, the monomials in $P_{k, +}$ cannot cancel with any monomial in  $P_{k, -}$ so the last equality implies that $P_{k, +}=P_{k, -}=0$. In conclusion, each $P_k$ equals zero and so $P$ equals zero. Since $f\equiv P\pmod J$, we obtain $f\in J$ and then $I\subseteq J$, concluding the proof.
	\end{proof}

	\begin{remark}\label{rem:homo}
		In the proof of Theorem~\ref{thm:main}, we assumed that $P$ is multilinear of degree $n$ just to simplify the notations and computations. Indeed, identities (2) and (3) imply $y_{1, 1}^2, z_{1, 1}^2\in J$. Therefore, following the proof of Theorem~\ref{thm:main}, we see that if $P$ is not multilinear, we only need to replace the products
		$(YZ)_{S, T}$ by more generic products 
		$$y_{s_1, 0}^{e_1}\cdots y_{s_j, 0}^{e_j}\cdot z_{t_1, 0}^{\varepsilon_1}\cdots z_{t_k, 0}^{\varepsilon_k},$$
		where $e_j, \varepsilon_k\ge 1$ and allow $U, V\subseteq\{1, \ldots, n\}$ to be multisets in the products $(Y|Z)_{U, V}, (Z|Y)_{V, U}, \{Y, Z\}_{U, V}$ and $\{Z, Y\}_{V, U}$. The proof follows similarly by writing $P_i=P_{i, +}+P_{i, -}$, where $P_{i, +}$ and $P_{i, -}$ only contain monomials with the sum $\varepsilon_1+\cdots+\varepsilon_k$ being even and odd, respectively. Lemma~\ref{lem:com} is still applicable for $R\in \{P_{i, \pm}, P_{i, +}\pm P_{i, -}\}$ in this case:  the only difference is that $R^C$ is not necessarily multilinear.  
    Thus Theorem~\ref{thm:main} also holds for any infinite field $F$ of characteristic $\ne 2$. 
	\end{remark}
	
	As an application of the previous theorem, we obtain the following result.

	\begin{proposition}\label{prop:cod}
		For each $n\ge 1$, we have 
		\begin{equation}\label{eq:binom}
			c_n^*(\mathbb M)=2^{n}\left(2\sum_{j=0}^{n}\binom{n}{j}\binom{j}{\lfloor\frac{j}{2}\rfloor}2^{-j}-1\right).
		\end{equation}
		In particular, $c_n^*(\mathbb M)\approx 4^nn^{-1/2}$ and $\mathrm{exp}^*(\mathbb M)=4$.
	\end{proposition}
	
	\begin{proof}
		
		By Eq.~\eqref{codim}, we have  $$c_n^*(\mathbb M)=\displaystyle \sum_{\langle n\rangle}\binom{n}{{\langle n\rangle}}c_{\langle n\rangle}(\mathbb M).$$
		%For each composition $\langle n\rangle$ of $n$ as $n=n_1+n_2+n_3+n_4$ with $n_i\ge 0$, 
		%let $P_{\langle n\rangle}$ be the space of all multilinear polynomials of degree $n$ in the $n$ variables $y_{i, 0}, z_{j, 0}, y_{k, 1}$ and $z_{l, 1}$ with $i\le n_1, j\le n_2, k\le n_3$ and $l\le n_4$. It is clear that \begin{equation}\label{eq:cod}
		%	c_n(M_{1, 1})=\sum_{\langle n\rangle}\binom{n}{\langle n\rangle}\dim \frac{P_{\langle n\rangle}}{P_{\langle n\rangle}(M_{1, 1})}.   
		%\end{equation}
		From the proof of Theorem~\ref{thm:main}, we know that $P_{\langle n\rangle}(\mathbb M)$ is the null space unless $|n_3-n_4|\le 1$. 
		
		If $n_3=n_4=k\ge 0$, we follow the proof of Theorem~\ref{thm:main} and the notations there and conclude that the quotient space $P_{\langle n\rangle}(\mathbb M)$ is generated by the following linearly independent monomials:
		$$y_{1, 0}\cdots y_{n_1, 0}z_{1, 0}\cdots z_{n_2, 0}\cdot y_{1, 1}z_{1, 1}\cdots y_{k, 1}z_{k, 1},$$
        and
        $$y_{1, 0}\cdots y_{n_1, 0}z_{1, 0}\cdots z_{n_2, 0}\cdot z_{1, 1}y_{1, 1}\cdots z_{k, 1}y_{k, 1}.$$ In particular, the corresponding space has dimension $2$ if $k>0$ and dimension $1$ if $k=0$. Therefore, from Eq.~\eqref{codim}, the case $n_3=n_4$ contributes with
		\begin{align*}
			S_1 &=\sum_{n_1+n_2=n}\frac{n!}{n_1!n_2!}+2\sum_{0< k\le n/2}\sum_{n_1+n_2=n-2k}\frac{n!}{n_1!n_2!(k!)^2}\\&=\sum_{j=0}^n\binom{n}{j}+2\sum_{0< k\le n/2}\frac{n!}{(n-2k)!(k!)^2}\sum_{j=0}^{n-2k} \binom{n-2k}{j}.    
		\end{align*}
		From the identity $\sum_{j=0}^{n-2k} \binom{n-2k}{j}=2^{n-2k}$, we obtain
		$$S_1=2^n+2\sum_{0< k\le n/2}\binom{n}{2k}\binom{2k}{k}2^{n-2k}=-2^n+2\sum_{0\le k\le n/2}\binom{n}{2k}\binom{2k}{k}2^{n-2k}.$$
		Similarly, for the two cases $|n_3-n_4|= 1$ the corresponding quotient spaces have dimension $1$ and, combined, they contribute with 
		$$S_2=2\sum_{0\le k<n/2}\binom{n}{2k+1}\binom{2k+1}{k}2^{n-2k-1}.$$
		Since $c_n^*(\mathbb M)=S_1+S_2$, we directly derive Eq.~\eqref{eq:binom}. We now proceed to the asymptotic growth of $c_n^*(\mathbb M)$. Stirling's formula yields the asymptotic formula $\binom{2k}{k}\sim \frac{4^k}{\sqrt{\pi k}}$, so it suffices to prove that
		$$\Delta_n:=\sum_{j=1}^n\binom{n}{j}j^{-1/2}\approx  2^nn^{-1/2}.$$
		We clearly have the lower bound $\Delta_n\ge n^{-1/2}\sum_{j=1}^n\tbinom{n}{j}=n^{-1/2}(2^n-1)$. For the upper bound, it is direct to verify that
		Stirling's formula yields $\binom{n}{j}<\binom{n}{\lfloor n/3\rfloor}<c\cdot 1.9^n$ if $j<n/3$ and $n$ is large enough, where $c$ is an absolute constant. Hence $\sum_{j<n/3}\binom{n}{j}j^{-1/2}<cn\cdot 1.9^n$ and so
		
		\begin{align*}
		\Delta_n &<cn \cdot 1.9^n+ \sum_{n/3\le j\le n}\binom{n}{j}j^{-1/2}\\ {} &\le  cn\cdot1.9^n+\left(\frac{n}{3}\right)^{-1/2}\sum_{n/3\le j\le n}\binom{n}{j}
        \\ {} &< Cn^{-1/2}2^n,\end{align*}
		for some absolute constant $C$. 
	\end{proof}

    Next, we will deal with the decomposition (\ref{ncocharacter}) of $\chi_{\langle n\rangle}(\mathbb M)$. Remark \ref{multiplicities} implies that, for every composition $\langle n \rangle=(n_1,n_2,n_3,n_4)$ of $n$ and every multipartition $\langle \lambda\rangle\vdash \langle n \rangle$ having nonzero multiplicity, the corresponding Young multitableaux $T_{\langle \lambda\rangle}$ satisfies the following property: all the tableaux in $T_{\langle \lambda\rangle}$ have at most 1 row. It is clear that the $n_1$ symmetric variables of degree 0 and the $n_2$ skew variable of degree 0 do not affect the highest weight vector associated to $T_{\langle \lambda\rangle}$, modulo $Id^*(\mathbb M).$ Thus, the multiplicities $m_{\langle \lambda\rangle}$ depend only on the values of $n_3$ and $n_4.$ 
    
    From the proof of Proposition \ref{prop:cod}, we obtain the following. If $n_3=n_4>0,$ then there are two highest weight vectors that are linearly independent modulo $Id^*(\mathbb M)$, but any three of such vectors are linearly dependent modulo $Id^*(\mathbb M)$. Consequently, $m_{\langle \lambda\rangle}=2$ in this case. If $|n_3-n_4|=1$ or $n_3=n_4=0$ then, modulo $Id^*(\mathbb M)$, there is (up to a scalar) a unique highest weight vector that is not a $\ast$-identity, and hence $m_{\langle \lambda \rangle}=1$. In all other cases, $m_{\langle \lambda \rangle}=0$. This yields the following theorem.

    \begin{theorem}\label{thm:coc}
        Let $F$ be a field of characteristic zero and consider the decomposition $\chi_{\langle n\rangle}(\mathbb M)=\displaystyle \sum_{\langle \lambda \rangle\vdash \langle n\rangle} m_{\langle \lambda \rangle}\chi_{\langle \lambda \rangle}$ of the $\langle n\rangle$-cocharacter of $\mathbb M$. Then, we have  $$ m_{\langle \lambda \rangle}=\begin{cases}
         2, &\, \text{if}\;\,  n_3=n_4>0;\\
         1, & \, \text{if}\;\, |n_3-n_4|=1\;\, \text{or}\;\, n_3=n_4=0;\\
         0, & \, \text{if}\;\,  |n_3-n_4|>1.
        \end{cases}$$
    \end{theorem}

\section{Central $*$-Polynomials}
	
	In this section, we aim to present the generators of the $T_2^*$-space $Id^{z,*}(\mathbb M)$. Let $W$ be the $T_2^*$-space generated by $y_{1, 0}$ and by the polynomials in $I=Id^*(\mathbb M)$.
	In order to state the result, we shall present two technical lemmas.
	
	\begin{lemma}\label{lem:aux}
		The set $W$ is closed under products and contains the polynomials $z_{1, 0}z_{2, 0}$ and $y_{1,1}\circ z_{1, 1}$. Moreover, for sets $U, V\subseteq \{1, \ldots, n\}$ with $|U|=|V|>0$, $W$ also contains the polynomials
		$$(Y|Z)_{U, V}+(Z|Y)_{V, U}\quad\text{and}\quad z_{1, 0}\left( (Y| Z)_{U, V}-(Z|Y)_{V, U}\right).$$
		Here, $(Y| Z)_{U, V}$ and $(Z|Y)_{U, V}$ are as in Definition~\ref{def:basic}.
	\end{lemma}
	
	\begin{proof}
		Recall that $W$ is the $T_2^*$-space generated by $I$ and $y_{1, 0}$. In particular, in order to prove that $W$ is closed under products, it suffices to prove a finite product of variables $y_{i, 0}$ belongs to $W$. However, the latter follows directly by the fact that any such product is again symmetric and of homogeneous degree zero $\pmod I$.

		Now we observe that, by Remark \ref{consequences}, $y_{1, 1}\circ z_{1, 1}\in (\mathcal{F}^{(0)})^+$ and $z_{1, 0}z_{2, 0}$ is symmetric and of homogeneous degree zero $\pmod{I}$, hence they are in $W$. Since $W$ is closed under products, if $1\le u_1<\cdots<u_k\le n$ and $1\le v_1<\cdots <v_k\le n$ are the increasing sequences of the elements in $U$ and $V$, we obtain 
		$$P=\prod_{l=1}^k(y_{u_l, 1}\circ z_{v_l, 1})\in W.$$
		However, since $z_{1, 1}z_{2, 1}, y_{1, 1}y_{2, 1}\in I$, a direct calculation yields 
		$$P\equiv (Y|Z)_{U, V}+(Z|Y)_{V, U}\pmod I.$$
		Therefore, $(Y|Z)_{U, V}+(Z|Y)_{V, U}\in W$. 
		Again, by Remark \ref{consequences}, $[y_{1, 1},z_{1, 1}]\in (\mathcal{F}^{(0)})^-$. Since $z_{1, 0}z_{2, 0}\in W$, we obtain 
		$z_{1, 0}[y_{1, 1},z_{1, 1}]\in W$.  Similarly to the previous case, we obtain 
		$$Q=z_{1, 0}[y_{u_1, 1},z_{v_1, 1}]\prod_{l=2}^k(y_{u_l, 1}\circ z_{v_l, 1})\in W,$$
		and $Q\equiv z_{1, 0}((Y|Z)_{U, V}-(Z|Y)_{V, U})\pmod I$. In conclusion, $z_{1, 0}((Y|Z)_{U, V}-(Z|Y)_{V, U})\in W$. 
	\end{proof}

    \begin{lemma}\label{lem:sum}
    Let $P, Q\in \mathcal{F}$ such that $P$ only has monomials of the form $\alpha (YZ)_{S, T}(Y|Z)_{U, V}$ and $Q$ only has monomials of the form $\alpha (YZ)_{S, T}(Z|Y)_{U, V}$. If $P^C=Q^C$, then $P+Q$ is a sum of polynomials of the form 
    $$\alpha\cdot  (YZ)_{S, T}((Y|Z)_{U, V}+(Z|Y)_{V, U}).$$
    On the other hand, if $P^C=-Q^C$, then $P+Q$ is a sum of polynomials of the form 
    $$\alpha\cdot  (YZ)_{S, T}((Y|Z)_{U, V}-(Z|Y)_{V, U}).$$    
\end{lemma}

\begin{proof}
  Let $\alpha, \beta \in F$ and consider the monomials 
$M_1=\alpha\cdot (YZ)_{S,T}(Y|Z)_{U, V}$ and $M_2=\beta\cdot  (YZ)_{S', T'}(Z|Y)_{V', U'}$. Observe that $M_1^C=M_2^C$ if and only if $\alpha=\beta$ and $S=S', T=T', U=U'$ and $V=V'$. In this case, we have $M_1+M_2=\alpha\cdot (YZ)_{S, T}((Y|Z)_{U, V}+(Z|Y)_{V, U})$. The latter combined with our assumption on $P$ and $Q$ proves te first statement. The second statement follows similarly by looking at the equality $M_1^C=-M_2^C$.
\end{proof}

We are now ready to present the following result.

\begin{theorem}\label{thm:main2}
    Let $F$ be a field of characteristic zero. Then  $Id^{z,*}(\mathbb M)=W$.
\end{theorem}

\begin{proof}
Let $P$ be a central polynomial. Since $F$ has characteristic zero, we may assume that $P$ is multilinear of degree $n$. Moreover, observe that any polynomial $Q$ with $P\equiv Q\pmod I$ is also central. In particular, we can suppose that $P=P_0+P_1+P_2+P_3+P_4$, where the monomials appearing in each $P_i$ are restricted as in the proof of Theorem~\ref{thm:main}. We can also take the decomposition $P_i=P_{i, +}+P_{i, -}$ as in the proof of Theorem~\ref{thm:main}.

 Take $y_{i, 0}=a_i(e_{11}+e_{22}), z_{i, 0}=b_i(e_{11}-e_{22})$ and $y_{i, 1}=z_{i, 1}=0$ for each $1\le i\le n$. We obtain
$$Q=(Q_{0, +}+Q_{0, -})e_{1 1}+(Q_{0, +}-Q_{0, -})e_{2 2},$$
where $Q_{0, \pm}\in F$ is the image of $P_{0, \pm}^C$ at the elements $a_i, b_i$. Since $Q$ is central, Remark~\ref{rem:center} provides $Q_{0, +}+Q_{0, -}=Q_{0, +}-Q_{0, -}$ and so $Q_{0, -}=0$. From Lemma~\ref{lem:com}, $P_{0, -}$ equals zero. In other words, $P_{0}=P_{0, +}$ only  contains monomials with an even number of variables $z_{i, 0}$. The latter combined with Lemma~\ref{lem:aux} entails that $P_0\in W$. In particular, we can assume without loss of generality that $P_0=0$. 

Now take $y_{i, 0}=a_i(e_{11}+e_{22}), z_{i, 0}=b_i(e_{11}-e_{22}), y_{i, 1}=c_ie_{21}$ and $z_{i, 1}=d_ie_{12}$ for each $1\le i\le n$.
Similarly to the previous computation, we obtain 
$$R=(Q_{1, +}-Q_{1, -})e_{22}+(Q_{2, +}+Q_{2, -})e_{11}+(Q_{3, +}-Q_{3, -})e_{21}+(Q_{4, +}+Q_{4, -}) e_{12},$$
where $Q_{k, \pm}\in F$ is the image of $P_{k, \pm}^C$ at the elements $a_i, b_i, c_i, d_i$. Since  $R$ is central, we obtain $Q_{3, +}-Q_{3, -}=0=Q_{4, +}+Q_{4, -}$. As in the proof of Theorem~\ref{thm:main}, we conclude that $P_3=P_4=0$. Therefore, 
\begin{equation}\label{eq:part}
    P=P_{1, +}+P_{1, -}+P_{2, +}+P_{2, -}.
\end{equation}

Using again the fact that $R$ is central, we also obtain $Q_{1, +}-Q_{1, -}=Q_{2, +}+Q_{2, -}$. From Lemma~\ref{lem:vanish}, we conclude that 
$$P_{1, +}^C-P_{1, -}^C=P_{2, +}^C+P_{2, -}^C.$$
Comparing the parity of the number of variables $z_{i, 0}$ in the last equality we obtain $P_{1, +}^C=P_{2, +}^C$ and $P_{1, -}^C=-P_{2, -}^C$. From Lemma~\ref{lem:sum}, the equality $P_{1, +}^C=P_{2, +}^C$ entails that $P_{1, +}+P_{2, +}$ is a sum of polynomials of the form 
$$\alpha y_{i_1, 0}\cdots y_{i_k, 0}z_{j_{1}, 0}\cdots z_{j_{l}, 0}((Y|Z)_{U, V}+(Z|Y)_{V, U}).$$
Moreover, from construction, $l$ is even. From Lemma~\ref{lem:aux} it is direct to verify that each such polynomial is in $W$, hence $P_{1, +}+P_{2, +}\in W$. 

Again, using Lemma~\ref{lem:sum}, the equality $P_{1, -}^C=-P_{2, -}^C$ entails that $P_{1, -}+P_{2, -}$ is a sum of polynomials of the form 
$$\alpha y_{i_1, 0}\cdots y_{i_k, 0}z_{j_{1}, 0}\cdots z_{j_{l}, 0}((Y|Z)_{U, V}-(Z|Y)_{V, U}).$$
Moreover, from construction, $l$ is odd. From Lemma~\ref{lem:aux} it is direct to verify that each such polynomial is in $W$, hence $P_{1, -}+P_{2, -}\in W$. From Eq.~\eqref{eq:part}, we conclude that $P\in W$.
\end{proof}

\begin{remark}
    Following Remark~\ref{rem:homo} and the proof of Theorem~\ref{thm:main2}, we also conclude that Theorem~\ref{thm:main2}  holds for any infinite field of characteristic $\ne 2$. 
\end{remark}

\subsection*{Declaration of competing interest}

The authors declare that they have no known competing financial interests or personal relationships that could have appeared to influence the work reported in this paper.

%\section*{Data availability statement}No data are associated with this article.


\begin{thebibliography}{999}

%\bibitem{Plamen} J. Colombo and P.~E. Koshlukov, Central polynomials in the matrix algebra of order two, Linear Algebra Appl. {\bf 377} (2004), 53--67; MR2021602

%\bibitem{Plamen1} J. Colombo and P.~E. Koshlukov, Identities with involution for the matrix algebra of order two in characteristic $p$, Israel J. Math. {\bf 146} (2005), 337--355; MR2151607

\bibitem{Bahturin} Y. Bahturin, M.~V. Tvalavadze and T.~V. Tvalavadze, Group gradings on superinvolution simple superalgebras, Linear Algebra Appl. {\bf 431} (2009), no.~5-7, 1054--1069; MR2535573

\bibitem{Plamen3} A.~P. Brand\~ao Jr., P.~E. Koshlukov and A.~N. Krasilnikov, Graded central polynomials for the matrix algebra of order two, Monatsh. Math. {\bf 157} (2009), no.~3, 247--256; MR2520727

\bibitem{Plamen2}  A.~P. Brand\~ao Jr. and P.~E. Koshlukov, Central polynomials for ${\Bbb Z}_2$-graded algebras and for algebras with involution, J. Pure Appl. Algebra {\bf 208} (2007), no.~3, 877--886; MR2283432

\bibitem{Ana} J.~P. Cruz and A.~C. Vieira, Central polynomials of the second-order matrix algebra with graded involution, J. Algebra {\bf 639} (2024), 574--595; MR4667607

%\bibitem{Diogo} D. Diniz, E. P. da Fonsêca and L. F. Ramos, Graded polynomial identities for all group gradings on $M_2(F)$. (2026, preprint).

\bibitem{Giam1} V.~S. Drensky and A. Giambruno, Cocharacters, codimensions and Hilbert series of the polynomial identities for $2\times 2$ matrices with involution, Canad. J. Math. {\bf 46} (1994), no.~4, 718--733; MR1289056

\bibitem{Drensky1} V.~S. Drensky, A minimal basis for identities of a second-order matrix algebra over a field of characteristic $0$, Algebra i Logika {\bf 20} (1981), no.~3, 282--290, 361; MR0648317

\bibitem{livro} V.~S. Drensky, {\it Free algebras and PI-algebras}, Springer, Singapore, 2000; MR1712064

\bibitem{Formanek} E.~W. Formanek, P. Halpin and W.-C.~W. Li, The Poincar\'e{} series of the ring of $2\times 2$\ generic matrices, J. Algebra {\bf 69} (1981), no.~1, 105--112; MR0613860

\bibitem{Shest} C. G\'omez-Ambrosi, J.~A. Laliena~Clemente and I.~P. Shestakov, On the Lie structure of the skew elements of a prime superalgebra with superinvolution, Comm. Algebra {\bf 28} (2000), no.~7, 3277--3291; MR1765316

\bibitem{Ninni} A. Ioppolo, The exponent for superalgebras with superinvolution, Linear Algebra Appl. {\bf 555} (2018), 1--20; MR3834188

\bibitem{Lev} D.~V. Levchenko, Finite basis property of identities with involution of a second-order matrix algebra, Serdica {\bf 8} (1982), no.~1, 42--56; MR0680769

\bibitem{Racine} M.~L. Racine, Primitive superalgebras with superinvolution, J. Algebra {\bf 206} (1998), no.~2, 588--614; MR1637088

\bibitem{Regev} A. Regev, Existence of identities in $A\otimes B$, Israel J. Math. {\bf 11} (1972), 131--152; MR0314893

\bibitem{Regev2} A. Regev, Codimensions and trace codimensions of matrices are asymptotically equal, Israel J. Math. {\bf 47} (1984), no.~2-3, 246--250; MR0738172

\bibitem{Okt} S.~V. Okhitin, Moscow Univ. Math. Bull. {\bf 43} (1988), no.~4, 49--51; translated from Vestnik Moskov. Univ. Ser. I Mat. Mekh. {\bf 1988}, no.~4, 61--63; MR0972718

\end{thebibliography}
\end{document}